\documentclass{amsart}

\usepackage[utf8]{inputenc}
\usepackage{tikz}
\usepackage{amsfonts, amssymb, amsmath, amsthm}
\usepackage{graphicx}
\usepackage{float, comment}
\usepackage[hidelinks]{hyperref}
\usepackage{verbatim}
\usepackage{mathtools, nccmath}

\usepackage{enumitem}
\usepackage{verbatim}

\usepackage{centernot}

\newcommand{\notiff}{%
  \mathrel{{\ooalign{\hidewidth$\not\phantom{"}$\hidewidth\cr$\iff$}}}}

\newtheorem{teor}{Theorem}[section]
\newtheorem{lema}[teor]{Lemma}
\newtheorem{prop}[teor]{Proposition}
\newtheorem{fact}[teor]{Fact}

\newtheorem{defin}[teor]{Definition}
\newtheorem*{claim}{Claim}
\newtheorem*{prop*}{Proposition}
\newtheorem{notation}{Notation}
\newtheorem{cor}[teor]{Corollary}
\newtheorem{counterexample}[teor]{Counterexample}
\newtheorem{remark}[teor]{Remark}

\newtheorem{external claim}[teor]{Claim}
\newtheorem{example}[teor]{Example}

\newcommand{\C}{{\mathfrak C}}
\newcommand{\R}{\mathbb{R}}
\newcommand{\mL}{\mathcal{L}}
\newcommand{\I}{\mathcal{I}}
\newcommand{\II}{\mathbf{I}}

\DeclareMathOperator{\tp}{{tp}}
\DeclareMathOperator{\qftp}{{tp^{qf}}}
\DeclareMathOperator{\emtp}{{EMtp}}
\DeclareMathOperator{\Th}{{Th}}
\DeclareMathOperator{\aut}{{Aut}}

\DeclareMathOperator{\age}{{Age}}
\DeclareMathOperator{\Ind}{{Ind}}

\newcommand{\orcidlogo}{\includegraphics[height=\fontcharht\font`\B]{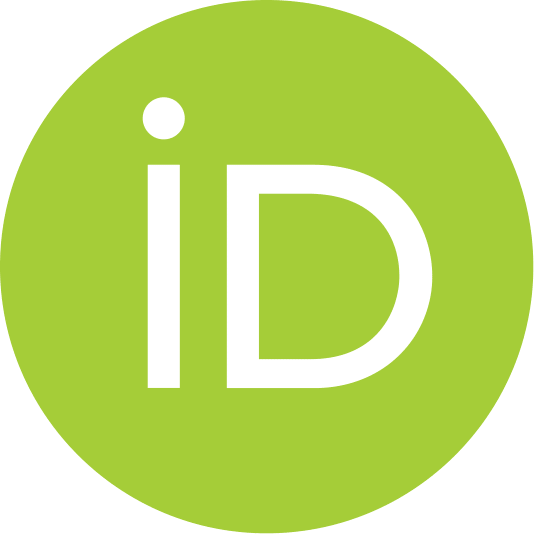}}
\newcommand{\orcid}[1]{\href{#1}{\orcidlogo #1}}

\usepackage{stackengine}
\stackMath

\usepackage[
backend=bibtex,
style=alphabetic,
]{biblatex}
\title{$n$-dependent continuous theories and hyperdefinable sets }
\author{Adrian Portillo Fernández}
\address{Adri\'{a}n Portillo \orcid{https://orcid.org/0000-0001-9354-8574}}
\email{Adrian.Portillo-Fernandez@math.uni.wroc.pl}
\date{January 2024}
\thanks{\noindent The author was supported by the Narodowe Centrum Nauki grant no. 2016/22/E/ST1/00450.}

\address{Instytut Matematyczny Uniwersytetu Wroc{\l}awskiego, pl. Grunwaldzki 2, 50-384 Wroc{\l}aw, Poland}


\begin{document}
\begin{abstract}
    We define the continuous modeling property for first-order structures and show that a first-order structure has the continuous modeling property if and only if its age has the embedding Ramsey property. We use generalized indiscernible sequences in continuous logic to study and characterize $n$-dependence  for continuous theories and first-order hyperdefinable sets in terms of the collapse of indiscernible sequences.
\end{abstract}
\maketitle
\section{Introduction}

Continuous model theory is a growing area of model theory that has been developing very fast in recent years. Many of the most important dividing lines for first-order theories have been also defined for continuous theories (Stability \cite{MR2657678, MR2723787}, NIP \cite{ContVCclasses}, Distality \cite{anderson2023fuzzy}). One invaluable tool for the characterization of dividing lines in first-order theories is (generalized) indiscernible sequences (See \cite[Chaper VII]{Shelah1982-SHECTA-5}, \cite{SCOW20121624}, \cite{Chernikov2014OnN}, \cite{Guingona2019}). We present natural continuous counterparts of generalized indiscernibles and the modeling property (where the index structures are still first-order) and show that a first-order structure has the continuous modeling property (see Definition \ref{definition: CMP}) if and only if its age has the embedding Ramsey property (Theorem \ref{CMP iff Ramsey}). Several notions around this topic have been also defined in positive logic. Dobrowolski and Kamsma (see \cite{Dobrowolski2021KimindependenceIP}) proved that $s$-trees have the (positive logic version of) modeling property in thick theories, later in  \cite[Theorems 1.2, 1.2 and 1.3]{kamsma2023positive} it was shown that $str$-trees, $str_0$-trees (the reduct of $str$-trees that forgets the length comparison relation) and arrays also have the modeling property in positive thick theories.
 
The notion of a dependent theory was first introduced by Shelah in \cite{Shelah1982-SHECTA-5}. 
In later work \cite{Shelah2005StronglyDT, Shelah2007DefinableGF}, Shelah introduced the more general notion of $n$-dependence. This notion was studied in depth in \cite{Chernikov2014OnN}, where the authors give a characterization of $n$-dependent theories in terms of the collapse of indiscernible sequences (See \cite[Theorem 5.4]{Chernikov2014OnN}). In the continuous context, the definition of $n$-dependence was introduced in \cite{Chernikov2020HypergraphRA} using a generalization of the $VC_n$ dimension. Section 10 of the aforementioned paper is dedicated to several operations which preserve $n$-dependence.  The proof of \cite[Theorem 5.4]{Chernikov2014OnN} contains a mistake \footnote{After sending the manuscript to the authors of \cite{Chernikov2014OnN}, they acknowledged that there is a mistake and proposed a short correction that we discuss at the end of Section \ref{section:n-dep}} in the implication $(3)\implies (2)$; we provide a counterexample to the key claim (see Counterexample \ref{counterexample}). Using  the tools developed in Section \ref{section: modelling property} we give an alternative proof of the theorem, obtaining generalizations of \cite[Theorem 5.4]{Chernikov2014OnN} to continuous logic theories (Theorem \ref{n-dep and collapsing intro}) and hyperdefinable sets (Theorem \ref{ndep an collapsing hyperim intro}).

Let $G_{n+1,p}$ be the Fraïssé limit of the class of ordered $(n+1)$-partite $(n+1)$-uniform hypergraphs and $G_{n+1}$ be the Fraïssé limit of the class of ordered $(n+1)$-uniform hypergraphs. By a $G_{n+1,p}$-indiscernible sequence $(a_g)_{g_\in G_{n+1,p}}$ we mean that for any $W,W'\subseteq G_{n+1,p}$ if the quantifier free types of $W$ and $W'$ coincide (in the language $\mL_{opg}$ defined in Section \ref{section:n-dep}), then the types of the tuples $(a_g)_{g\in W}$ and $(a_g)_{g\in W'}$ also coincide (similarly for $G_{n+1}$-indiscernibility, see Definition \ref{defin: gen. indisc.}). We say that the sequence $(a_g)_{g_\in G_{n+1,p}}$ is $\mL_{op}$-indiscernible if for any $W,W'\subseteq G_{n+1,p}$ with the same quantifier free type in the language $\mL_{op}$ (see Section \ref{section:n-dep}), the types of the tuples $(a_g)_{g\in W}$ and $(a_g)_{g\in W'}$ also coincide. 
\begin{teor} \label{n-dep and collapsing intro}
    Let $T$ be a complete continuous logic theory. The following are equivalent:
    \begin{enumerate}
        \item $T$ is $n$-dependent.
        \item Every $G_{n+1,p}$-indiscernible is $\mL_{op}$-indiscernible.
        \item Every $G_{n+1}$-indiscernible is order-indiscernible.
    \end{enumerate}
\end{teor}
  

We introduce the following definition of $n$-dependent hyperdefinable sets:
\begin{defin} The hyperdefinable set $X/E$ has the $n$-independence property, $IP_n$ for short, if for some $m<\omega$ there exist two distinct complete types $p,q\in S_{X/E\times \C^m}(\emptyset)$ and a sequence $(a_{0,i},\dots, a_{n-1,i})_{i< \omega}$ such that for every finite $w\subset \omega^n$ there exists $b_w\in X/E$ such that $$\tp(b_w, a_{0,i_0},\dots, a_{n-1,i_{n-1}}  )=p \iff (i_0,\dots, i_{n-1})\in w$$ $$\tp(b_w, a_{0,i_0},\dots, a_{n-1,i_{n-1}}  )=q \iff (i_0,\dots, i_{n-1})\notin w. $$ 
\end{defin}
Using the tools developed in Sections \ref{section: modelling property} and \ref{section:n-dep}, we prove the theorem below, 
 which is a counterpart for hyperdefinable sets of Theorem \ref{n-dep and collapsing intro}. Here, $\Psi^{n+1}_{\mathcal{F}_{X/E}}$  is a special family of functions (see Notation \ref{notation PSI} in Section \ref{section: hyperdef n-dep}).
\begin{teor}\label{ndep an collapsing hyperim intro}
 The following are equivalent:
    \begin{enumerate}
        \item $X/E$ is $n$-dependent.
        \item Every $G_{n+1,p}$-indiscernible $(a_g)_{g\in G_{n+1,p}}$, where for every $g\in P_0(G_{n+1,p})$ we have $a_g\in X/E$, is $\mL_{op}$-indiscernible.
        \item For every $m\in \mathbb{N}$, every $G_{n+1}$-indiscernible with respect to $\Psi^{n+1}_{\mathcal{F}_{X/E}}$ sequence of elements of $\C^m\times X$  is order-indiscernible with respect to $ \Psi^{n+1}_{\mathcal{F}_{X/E}}$.
    \end{enumerate}
\end{teor}
Here $P_0(G_{n+1,p})$ is the first part of the partition of $G_{n+1,p}$, and by a $G_{n+1}$-indiscernible sequence $(a_g)_{g_\in G_{n+1}}$ with respect to $\Psi^{n+1}_{\mathcal{F}_{X/E}}$  we mean that for any $W,W'\subseteq G_{n+1,p}$ if the quantifier free types of $W$ and $W'$ coincide (in the language $\mL_{og}$ defined in Section \ref{section:n-dep}), then the $\Psi^{n+1}_{\mathcal{F}_{X/E}}$-types of the tuples $(a_g)_{g\in W}$ and $(a_g)_{g\in W'}$ also coincide.

Item $(3)$ of Theorem \ref{ndep an collapsing hyperim intro} cannot be improved by replacing indiscernibility with respect to $\Psi^{n+1}_{\mathcal{F}_{X/E}}$ by indiscernibility with respect to more general families of functions from $\mathcal{F}_{(\C^m \times X/E)^{n+1}}$ or by replacing $\mathcal{F}_{(\C^m \times X/E)^{n+1}}$ by $\Psi^{n'+1}_{\mathcal{F}_{X/E}}$ with $n'>n$ as showed by Example \ref{example optimal}.

\section{Preliminaries}\label{section:prelim}
\subsection{Continuous logic} We present some basic definitions and facts about continuous logic needed for this paper, we refer the reader to \cite{MR2657678} 
and \cite{MR2723787} for a more detailed exposition.

\begin{defin}
    A \emph{continuous (metric) signature} consists of: 
    \begin{itemize}
        \item A collection of function symbols f, together with their arity $n_f<\omega$.
        \item A collection of predicate symbols P, together with their arity $n_P<\omega$.
        \item A binary predicate symbol, denoted $d$, specified as the distinguished \emph{distance symbol}.
        \item For each $n$-ary symbol $s$ and $i<n$ a continuity modulus $\delta_{s,i}$, called the \emph{uniform continuity modulus of } $s$ with respect to the $i$th argument.
    \end{itemize}
\end{defin}

Given a continuous signature $\mL$, the collection of $\mL$-terms and atomic $\mL$-formulas are constructed as usual. In the continuous context, the quantifiers $\sup_x$ and $\inf_x$ play the roles of $\forall$ and $\exists$ respectively. The issue of connectives is a bit more delicate and we refer the reader to \cite{MR2657678} for an in depth treatment. Depending on the context, we will consider all uniformly continuous functions $u:[0,1]^n\to [0,1]$ for all $n<\omega$ as connectives or just some finite subset of such functions.

A \emph{condition} is an expression of the form $\varphi=0$ where $\varphi$ is a formula. Note that expressions of the form $\varphi\geq r$ and $\varphi\leq r$ can be expressed as conditions. A \emph{continuous $\mL$-theory} $T$ is a consistent set of $\mL$-conditions $\varphi=0$ where $\varphi$ is a sentence.

A \emph{complete $n$-type} in variables $x$ is a maximal satisfiable set of conditions with free variables $x_1,\dots,x_n\subseteq x$ for some $n$. The space of all types over the empty set in variables $x$ is denoted by $S_{x}$. If $x=(x_1,\dots,x_n)$, we denote $S_{x}$ by $S_{n}$. The space $S_{x}$ is a compact Hausdorff space when equipped with the finest topology for which all continuous formulas $\varphi$ are continuous functions $\varphi:S_{x}\to[0,1]$.


\begin{defin}
    A \emph{definable predicate} $f$ in variables $x$ is a continuous function $f:S_{x}\to[0,1]$.
\end{defin}

The following was proven in \cite[Proposition 3.4]{MR2657678} for a finite number of variables but the proof also applies for an infinite $x$.

\begin{fact}
    Definable predicates in variables $x$ can be uniformly approximated by continuous formulae in variables contained in $x$.
\end{fact}

By an abuse of notation, when results apply to both, we will usually also refer to definable predicates as formulas. We need to allow the domain of definable predicate to be an infinite Cartesian power of $\C$ to deal with hyperdefinable sets $X/E$ which are contained in an infinite product of sorts. 


 \subsection{The embedding Ramsey property}
Let $\mL'$ be a first-order language. Given $\mL'$-structures $A\subseteq B$, we write $\binom{B}{A}$ for the set of all embeddings from $A$ into $B$.   
\begin{defin}
    Let $A\subseteq B \subseteq C$ be $\mL'$-structures and let $r\in\mathbb{N}$. We write $$ C\to (B)^A_r$$ if for each coloring $\chi: \binom{C}{A}\to r$ there exists some $f\in \binom{C}{B}$ such that $\chi\upharpoonright_{f\circ \binom{B}{A}}$ is constant.
\end{defin}

\begin{defin}
    Let $\mL'$ be a first-order language and $\mathcal{C}$ be a class of finite $\mL'$-structures. We say that $\mathcal{C}$ has the \emph{embedding Ramsey property}, \emph{ERP} for short, if for every $A\subseteq B \in \mathcal{C}$ and $r<\omega$ there is $C\in \mathcal{C}$ such that $C\to (B)^A_r$.
\end{defin}

Note that the following holds:
\begin{fact}
    If a class of finite $\mL'$-structures has $ERP$, then all the structures of $\mathcal{C}$ are rigid (i.e. they have no nontrivial automorphisms).
\end{fact}

\begin{remark}
    Sometimes, the symbol $\binom{B}{A}$ is used to denote the set of all isomorphic copies of $A$ in $B$. If the class $\mathcal{C}$ consists of finite $\mL'$-structures which are rigid, then coloring embeddings from $A$ into $B$ is equivalent to coloring substructures $A\subseteq B$. 
\end{remark}
\section{Generalized indiscernibles}\label{section: modelling property}
Let $\mL'$ be a first-order language and $\mL$ be a continuous logic language. Unless specified otherwise, $T$ is a complete continuous $\mL$-theory with $\C\models T$ a monster model (i.e. $\kappa$-saturated and strongly $\kappa$-homogeneous for a strong limit cardinal $> |T|$) and $\I,\mathcal{J}$ are $\mL'$-structures. 

In this section, we present natural adaptations of the concepts of generalized indiscernibles and the modeling property to continuous logic and give a characterization of the continuous modeling property in the form of a continuous logic counterpart of \cite[Theorem 2.10]{SCOW2021102891}.

The following idea first appeared in \cite[Definition VIII.2.4]{Shelah1982-SHECTA-5}.

\begin{defin}\label{defin: gen. indisc.}
    Let $\II=(a_i: i\in \I)$ be an $\I$-indexed sequence, and let $A\subset \C$ be a small set of parameters.  We say that $\II$ is an $\I$-\emph{indexed indiscernible sequence over} $A$ if for all $n\in\omega$ and all sequences $i_1,\dots,i_n,j_1,\dots,j_n$ from $\I$ we have that 
    $$\qftp(i_1,\dots,i_n)=\qftp(j_1,\dots,j_n) \implies \tp({a}_{i_1}, \dots, {a}_{i_n}/A)=\tp({a}_{j_1}, \dots, {a}_{j_n}/A).$$
\end{defin}

We will refer to $\I$-indexed indiscernible sequences as $\I$-indiscernibles.

Next, we adapt the definition of \emph{locally based on} given in \cite{10.1215/00294527-3132797}. The first reference to this concept can be found in \cite{MZiegler}.

\begin{defin}[Locally based on] 
Let $\II=(a_i:i\in \I)$ be an $\I$-indexed sequence. We say that a $\mathcal{J}$-indexed sequence $(b_j:j\in\mathcal{J})$ is \emph{locally based on $\II$} if for any finite set of $\mL$ formulas $\Delta$, any finite tuple $\overline{j}\subseteq \mathcal{J}$ and $\varepsilon>0$ there is $\overline{i}\subseteq \I$ such that:
\begin{enumerate}
    \item $\qftp(\overline{i})=\qftp(\overline{j})$.
    \item $\lvert \varphi(b_{\overline{j}})- \varphi(a_{\overline{i}}) \rvert\leq \varepsilon$ for all $\varphi\in \Delta$.
\end{enumerate}
\end{defin}

 The original definition presented in \cite[Definition 2.5]{10.1215/00294527-3132797} is the following:
 \begin{defin}[Classical definition of Locally based on]
Let $\II=(a_i:i\in \I)$ be an $\I$-indexed sequence. We say that a $\mathcal{J}$-indexed sequence $(b_j:j\in\mathcal{J})$ is \emph{locally based on $\II$} if for any finite set of $\mL$ formulas $\Delta$, any finite tuple $\overline{j}\subseteq \mathcal{J}$ and $\varepsilon>0$ there is $\overline{i}\subseteq \I$ such that:
\begin{enumerate}
    \item $\qftp(\overline{i})=\qftp(\overline{j})$.
    \item $\tp^{\Delta}(b_{\overline{j}})= \tp^{\Delta}(a_{\overline{i}})$.
\end{enumerate}
\end{defin}
 
 Note that if we tried to use this stronger version of the property, it is easy to show that even for $\I=(\mathbb{N},<)$ we can find a sequence for which there are no indiscernible sequences locally based on it. Consider for example the theory $\Th([0,1],d)$ where $d$ is the distance predicate and the sequence $(1/n)_{n<\omega}$.

The next definition is then the natural continuous counterpart of \cite[Definition 2.17]{SCOW20121624}.

\begin{defin}[Continuous Modeling property]\label{definition: CMP}
    Given a continuous theory $T$, we say that \emph{$\I$-indexed indiscernibles have the continuous modeling property in $T$} if given any $\I$-indexed sequence $\II=(a_i:i\in \I)$ in a monster model $\C$ of $T$ there exists an $\I$-indiscernible sequence $(b_i:i\in\I)$ in $\C$ locally based on $\II$. We say that $\I$ has the \emph{continuous modeling property} if $\I$-indexed indiscernibles have the continuous modeling property in every continuous theory.
\end{defin}

As it is natural, if a first-order structure $\I$ has the continuous modeling property then it has the modeling property. More precisely:

\begin{prop} \label{CMP implies MP}
    Let $T$ be a first-order theory, and let $T'$ be its continuous logic counterpart (i.e., $T$ and $T'$ have the same models). Then $\I$ has the continuous modeling property in $T'$ if and only if $\I$ has the modeling property in $T$.
\end{prop}

\begin{proof}
    Clearly, If $\I$ has the continuous modeling property in $T'$ then it has the modeling property in $T$ since classical formulas are a subset of the $\{0,1\}$-valued continuous logic formulas.

    Assume now that $\I$ has the modeling property in $T$. Let $\II=(a_i:i\in \I)$ be any sequence in $\C\models T$. Since $\I$ has the modeling property, there is an $\I$-indiscernible sequence $(b_i:i\in \I)$ locally based on $\II$ (in the classical sense). We show that the sequence $(b_i:i\in \I)$ is locally based on $\II$ in our continuous logic sense. Since first-order formulas generate a dense subalgebra $\mathcal{A}$ of the set of all continuous logic formulas, for each continuous logic formula $f(x)$ and $\varepsilon>0$ there is $\varphi(x)\in \mathcal{A}$ such that $\lvert f(x)-\varphi(x)\rvert \leq \varepsilon/2$. Thus, for any tuples 
    $\overline{i},\overline{j}\subseteq \I$ 
    and tuples $a_{\overline{j}}, b_{\overline{i} }$ we have $$ \lvert f(a_{\overline{j}}) -f(b_{\overline{i}})\rvert \leq \lvert f(x)-\varphi(x)\rvert + \lvert \varphi(a_{\overline{j}}) -\varphi(b_{\overline{i}}) \rvert + \lvert f(x)-\varphi(x)\rvert \leq \lvert \varphi(a_{\overline{j}}) -\varphi(b_{\overline{i}}) \rvert +\varepsilon.$$
    Finally, note that by the definition of being locally based on (in the classical sense) for any $b_{\overline{i}}$ and finite $\Sigma \subset \mathcal{A}$, there is $\overline{j}$ with the same quantifier free type as $\overline{i}$ such that $ \varphi(b_{\overline{i}})=\varphi(a_{\overline{j}}) $ for every $\varphi\in \Sigma$. 
    Therefore, the sequence $(b_i:i\in \I)$ is locally based on $\II$ in the continuous sense.
\end{proof}

Next, we define two partial types that will be useful during this section. The first one is a generalization of the classical Ehrenfeucht-Mostowski type (EM-type for short). The second is a type whose realizations are exactly the $\I$-indiscernible sequences. They are based on \cite[Definitions 2.6 and 2.10]{SCOW20121624} respectively.

\begin{defin}
Let $\II=(a_i:i\in \I)$ be an $\I$-indexed sequence. The \emph{$EM$-type of $\II$} is the set of all conditions $\varphi(x_{i_1}\dots,x_{i_n})=0$ such that $\varphi(a_{j_1}\dots,a_{j_n})=0$ holds for every $j_1,\dots,j_n\in \I$ with the same quantifier free type as $i_1,\dots,i_n$. That is \begin{align*}
    \emtp(\II)(x_i: i\in \I)=\{& \varphi(x_{i_1},\dots,x_{i_n})=0 :  \varphi \in \mL, i_1,\dots,i_n\in \I  \\ &\text{ and for any } j_1,\dots,j_n\in \I  \text{ such that } \\ &\qftp(j_1,\dots,j_n)=\qftp(i_1,\dots,i_n),   \models\varphi(a_{j_1}\dots,a_{j_n})=0 \}.
\end{align*}
\end{defin}

\begin{defin}
    We define $\Ind(\I,\mL)$ as the following partial type:
    \begin{align*}
       \Ind(\I,\mL)(x_i:i\in \I):= \{& \varphi(x_{i_1},\dots,x_{i_n})= \varphi(x_{j_1},\dots,x_{j_n}):&\\
        & n<\omega, \overline{i}, \overline{j}\subseteq \I, \qftp(\overline{i})=\qftp(\overline{j}), \varphi(x_{i_1},\dots,x_{i_n})\in \mL\}.
    \end{align*}
\end{defin}

Finally, we define what it means for a partial type to be finitely satisfiable in a sequence. 

\begin{defin}[Finitely satisfiable] Let $\Gamma(x_i:i\in \I)$ be an $\mL$-type, and let $\II=(a_i:i\in \I)$ be an $\I$-indexed sequence. We say that $\Gamma$ is \emph{finitely satisfiable in $\II$} if for every finite $\Gamma_0\subseteq \Gamma^+$ and for every finite $A\subseteq \I$, there is $B\subseteq \I$, a bijection $f:A\to B$, and an enumeration $\overline{i}$ of $A$ such that:
$$ \qftp(\overline{i})=\qftp(f[\overline{i}])\text{ and } (a_{f(i)}:i\in A)\models \Gamma_0\upharpoonright \{x_i: i\in A\}. $$ 
Where $\Gamma^+:=\{ \varphi \leq 1/n: n<\omega ; \varphi=0\in \Gamma\}$.
\end{defin}

The following result gives a sufficient condition for the existence of an $\I$-indiscernible sequence locally based on $\II=(a_i:i\in \I)$.

\begin{lema}
\label{EMTP, based on and Ind(I,L)}
Let $\mathcal{J} \supseteq \I$ be $\mL'$-structures with the same age and let $\II=(a_i:i\in \I)$ be an $\I$-indexed sequence. 
\begin{enumerate}
    \item  A $\mathcal{J}$-indexed sequence $\mathbf{J}=(b_j:j\in \mathcal{J})$ is locally based on $\II$ if and only if $\emtp(\mathbf{J})\supseteq \emtp(\II)$.
    \item If $\Ind(\I,\mL)$ is finitely satisfiable in $\II$, then there is an $\I$-indexed indiscernible sequence $\Tilde{\II}:=(b_i:i\in \mathcal{I})$ locally based on $\II$.
\end{enumerate}
\end{lema}
\begin{proof}
    \begin{enumerate}
        \item Suppose $\mathbf{J}$ is locally based on $\II$. Fix $\varphi(x_{i_1},\dots,x_{i_n})=0\in \emtp_{\mL'}(\II)$ and let $\overline{i}=(i_1,\dots,i_n)$. If $ \varphi(x_{\overline{i}})=0$ is not in $\emtp(\mathbf{J})$, then $\varphi(b_{\overline{j}})\geq \varepsilon$ for some $\varepsilon>0$ and $\overline{j}\subseteq \mathcal{J}$ with the same quantifier free type as $\overline{i}$. By assumption, there is $\overline{i}'\subseteq \I$ satisfying the same quantifier free type as $\overline{j}$ and such that $\varphi(a_{\overline{i}'})\geq \varepsilon/2$, which contradicts $\varphi(x_{\overline{i}})\in \emtp(\II)$.

        Suppose now that $\emtp(\mathbf{J})\supseteq \emtp(\II)$. For a contradiction, assume that $\mathbf{J}=(b_j:j\in \mathcal{J})$ is not locally based on $\II$. That is, there is $\Delta \subseteq \mL$, $b_{\overline{j}}:=(b_{j_1},\dots,b_{j_n})$ from $\mathbf{J}$ and $\varepsilon>0$ such that there is no $\overline{i}\subseteq \I$ satisfying $\qftp(\overline{i})=\qftp(\overline{j})$ and $\lvert \varphi(b_{\overline{j}})- \varphi(a_{\overline{i}}) \rvert\leq \varepsilon$ for all $\varphi\in \Delta$. Let $\psi(x):=\max \{ \lvert  \varphi(x) - \varphi(b_{\overline{j}})\rvert : \varphi\in \Delta \}$, $\psi$ is a continuous logic formula and $\psi(b_{\overline{j}})=0$. By assumption, for any $\overline{i}\subseteq \I$ with the same quantifier free type as $\overline{j}$, $\psi(a_{\overline{i}})\geq \varepsilon$. Thus, $\psi(x)\geq \varepsilon \in \emtp(\II)$, which contradicts $\emtp(\mathbf{J})\supseteq \emtp(\II)$.
    
    \item Observe that if the type  $\Ind(\I,\mL)(x_i: i\in \I)$ 
    is finitely satisfiable in $\II$, then 
    $\Ind(\I,\mL) \cup \emtp(\II)$ is satisfiable. Let $\mathbf{J}\models \Ind(\I,\mL)\cup \emtp(\II)$. $\mathbf{J}$ is an $\I$-indiscernible sequence and is locally based on $\II$ by $(1)$.
    \end{enumerate}    
\end{proof}

We now prove the main result of this section. It is an extension of \cite[Theorem 2.10]{SCOW2021102891} to continuous logic. 

\begin{teor}\label{CMP iff Ramsey} Let $\mL'$ be a first-order language 
and let $\I$ be an infinite locally finite $\mL'$-structure
. Then, the following are equivalent:
\begin{enumerate}
    \item $\age(\I)$ has ERP.
    \item $\I$-indiscernibles have the continuous modeling property.
\end{enumerate}
\end{teor}
\begin{proof} 
    $(1)\implies (2)$. Assume $\age(\I)$ has ERP and let $\II= (a_i)_{i\in \I}$ be any $\I$-indexed sequence. Our goal is to prove that there exists $\mathbf{J}=(b_i)_{i\in \I}$ locally based on $\II$. By Lemma \ref{EMTP, based on and Ind(I,L)}, it is enough to show that $\Ind(\I,\mL)$ is finitely satisfiable in $\II$.

    Let $\Gamma_0\subset \Ind(\I,\mL)^+$ be any finite subset. For some $K,M<\omega$  $$\Gamma_0=\{ \lvert \varphi(x_{\overline{i}_p})-\varphi(x_{\overline{j}_p})\rvert <\frac{1}{n_m}: \qftp(\overline{i}_p)=\qftp(\overline{j}_p), \varphi\in \Delta, p<K,m<M  \}.$$ $\Gamma_0$ involves finitely many formulas $\Delta:=\{ \varphi_0,\dots,\varphi_m\}$, finitely many tuples $\overline{i}_p$, $\overline{j}_p$ and finitely many rationals $\frac{1}{n_m}$. Without loss of generality, we may assume that the formulas  $\varphi\in \Delta$ (and their tuples of variables) are of the form $\varphi((x_g)_{g\in A})$ for some $A\in \age(\I)$. Let $B\in \age(\I)$ be the structure generated by all the 
    coordinates of the tuples tuples $\overline{i}_p$, $\overline{j}_p$ involved in $\Gamma_0$. It is enough to prove the existence of a copy $B'$ of $B$ such that for any $\varphi((x_g)_{g\in A})\in \Delta$ and $A',A''\subseteq B$ copies of $A$,  $$\lvert \varphi((a_g)_{g\in A'}) - \varphi((a_g)_{g\in A''})\rvert \leq \frac{1}{n} $$ for some $n<\omega$ such that $1/n$ is smaller than any rational involved in $\Gamma_0$.

    If $\Delta$ involves only one formula $\varphi((x_g)_{g\in A})$, we proceed in the following manner: Linearly order the set of intervals $\{ [\frac{i}{n},\frac{i+1}{n}]: i<n \}$ and define an $n$-coloring of the copies $A'$ of $A$ by coloring each $A'$ with the first interval that contains $\varphi((a_g)_{g\in A'})$. Since $\age(\I)$ is Ramsey, we can find a copy $B'$ of $B$ homogeneous with respect to the coloring.
    Then, $(a_g)_{g\in B'}$ witnesses that $\Gamma_0$ is satisfied in $\II$. If $\Delta$ involves $k<\omega$ formulas $ \{ \varphi_i((x_g)_{g\in A_i}) : i<k\}$ and all the sets $A_i$ involved are isomorphic we can apply a similar trick, using as colors the hypercubes $$\{ [\frac{i_1}{n}, \frac{i_1+1}{n}]\times\cdots\times [\frac{i_k}{n}, \frac{i_k+1}{n}]: i_1,\dots,i_k<n \}.$$ 

    We claim that the latter is the only case we need to check. The proof is a standard argument in Ramsey theory which we sketch here for completeness
    
    Let $A_1\dots,A_m$ be structures in $\age(\I)$ and let $B\in \age(\I)$ embed every $A_i$ for $i<m$. Let $k_1,\dots,k_n$ be natural numbers and let $Z_n\in \age(\I)$ be such that $$Z_n\to (Z_{n-1})^{A_n}_{k_n}$$  for every $n<m$. We construct by induction a sequence of structures $Y_n\in \age(\I)$ for $0\leq n\leq m$.

    Case $n=0$: $Y_0=Z_m$

    Case $0<n<m$: By induction we have $Y_{n-1}\in \age(\I)$ isomorphic to $Z_{m-n+1}$. Color the copies of $A_{m-n+1}$ inside $Y_{n-1}$ with $k_{n-1}$ colors. By definition of $Z_{m-n+1}$, there is a copy $Y_n$ of $Z_{m-n}$ inside $Y_{n-1}$ such that all of the copies of $A_{m-n+1}$ contained in $Y_n$ have the same color.

Note that since $Y_n\subseteq Y_{n-1}$ for $0\leq n \leq m$ and all copies of $A_{m-n+1}$ inside $Y_n$ are of the same color, we have that $Y_n$ is homogeneous for copies of $A_j$ for all $m-n+1\leq j \leq m$. Therefore, $Y_m$ is homogeneous for all copies of $A_1,\dots,A_m$ and so it is the $B'$ we were looking for in the proof.

    $(2)\implies(1)$. 
    Let $A\subseteq B \in \age(\I)$ be arbitrary finite substructures of $\I$ and let  $\chi$ be a $k$-coloring of the embeddings of $A$ into $\I$. We expand $\I$ by adding a predicate $R_i$ for each fiber of the coloring. Let us denote this expanded structure by $\I'$ and this new language by $\mL$. Let $T$ be the $\mL$-theory of of $\I'$. Since $\I$ has the modeling property in $T$, there is an $\I$-indiscernible sequence $(b_i)_{i\in \I}$ locally based on $(i)_{i\in \I}$. Using the definition of locally based on for $\Delta:=\{ R_1,\dots,R_k \}$ we can find and embedding $f$ from $B$ into $\I$ such that \begin{align*}
    \qftp_{\mL'}(B)&=\qftp_{\mL'}(f[B])\\ &\text{ and}\\
    \tp^{\Delta}((b_g)_{g\in B})&=\tp^{\Delta}( (f(g))_{g\in B}).
    \end{align*}
    This implies that $\chi\upharpoonright_{f\circ \binom{B}{A}}$ is constant.
\end{proof}

In light of the previous theorem we will not make a distinction between continuous or classical modeling property from now on.

\section{Characterizing n-dependence through collapse of indiscernibles}\label{section:n-dep}
Let $T$ be a complete continuous $\mL$-theory with $\C\models T$ a monster model (i.e. $\kappa$-saturated and strongly $\kappa$-homogeneous for a strong limit cardinal $> |T|$). 

 In this section, we study $n$-dependent continuous formulae and give an analogous result to \cite[Theorem 5.4]{Chernikov2014OnN}. We present an alternative proof of the aforementioned result, which corrects a mistake made in the original source. 

The next few paragraphs contain basic facts about hypergraphs taken almost verbatim from \cite{Chernikov2014OnN}.

We work with three families of languages 
\begin{align*}
    \mL^n_{op}&=\{<, P_0(x),\dots, P_{n-1}(x)\},\\
    \mL^n_{og}&=\{<, R(x_0,\dots,x_{n-1})\},\\
    \mL^n_{opg}&=\{<, R(x_0,\dots,x_{n-1}), P_0(x),\dots, P_{n-1}(x)\}.
\end{align*}
  When $n<\omega$ is clear we will simply omit it. We consider the Ramsey classes of ordered $n$-uniform hypergraphs and ordered $n$-partite $n$-uniform hypergraphs.

An $\mL^n_{og}$-structure $(M,<,R)$ is an ordered $n$-uniform hypergraph if 
\begin{itemize}
    \item $(M,<)\models DLO$
    \item $R(a_0,\dots,a_{n-1})$ implies that $a_0,\dots,a_{n-1}$ are different,
    \item the relation $R$ is symmetric.
\end{itemize}

An $\mL^n_{opg}$-structure $(M,<,R,P_0,\dots,P_{n-1})$ is an ordered $n$-partite $n$-uniform hypergraph if
\begin{itemize}
    \item $M$ is the disjoint union $P_0\sqcup \dots \sqcup P_{n-1}$ such that if $R(a_0,\dots,a_{n-1})$ then $P_i\cap \{a_0,\dots,a_{n-1}\}$ is a singleton for every $i<n$,
    \item the relation $R$ is symmetric,
    \item $M$ is linearly ordered by $<$ and $P_0<\cdots<P_{n-1}$.
\end{itemize}

The following fact was proven in \cite{NESETRIL1977289}, \cite{NESETRIL1983183} and independently in \cite{1a7d4f5c-0dcd-37ec-b761-69bf07cad14f} for the case of nonpartite hypergraphs and in \cite{Chernikov2014OnN} for the case of partite hypergraphs.

\begin{fact} \label{fact: ordered hypergraphs are ramsey}
Let $K$ be the set of all finite ordered $n$-partite $n$-uniform hypergraphs and $\Tilde{K}$ be the set of all finite ordered $n$-uniform hypergraphs. The classes $K$ and $\Tilde{K}$ have the embedding Ramsey property.
\end{fact}

We will denote by $G_{n,p}$ the Fraïssé limit of $K$ and by $G_n$ the Fraïssé limit of $\Tilde{K}$.

\begin{remark}
    The theories of $G_n$ and $G_{n,p}$ can be axiomatized in the following way:
    \begin{enumerate}
        \item[1.] $(M, <, R)\models \Th(G_n)$ if and only if
        \begin{itemize}
            \item $(M,<)\models DLO$,
            \item $(M, <, R)$ is an ordered $n$-uniform hypergraph,
            \item For every finite disjoint sets $A_0,A_1\subset M^{n-1}$ such that $A_0$ consists of tuples with pairwise distinct coordinates and $b_0<b_1\in M$, there is $b\in M$ such that $b_0<b<b_1$ and $R(b,a_{i,1},\dots,a_{i,n-1})$ holds for every $(a_{0,1},\dots,a_{0,n-1})\in A_0$ and $\neg R(b,a_{1,1},\dots,a_{1,n-1})$ holds for every $(a_{1,1},\dots,a_{1,n-1})\in A_1$.
        \end{itemize}
        \item[2.] $(M, <, R, P_0,\dots, P_n)\models \Th(G_{n,p})$ if and only if
        \begin{itemize}
            \item For every $i<n$ $P_i(M)\models DLO$,
            \item $(M, <, R, P_0,\dots, P_n)$ is an ordered $n$-partite $n$-uniform hypergraph,
            \item for every $j<n$, finite disjoint sets $A_0,A_1\subset \prod_{i\neq j} P_i(M)$ and $b_0<b_1\in P_j(M)$ there is $b\in P_j(M)$ such that $b_0<b<b_1$ and $R(b,a_{i,1},\dots,a_{i,n-1})$ holds for every $(a_{0,1},\dots,a_{0,n-1})\in A_0$ and $\neg R(b,a_{1,1},\dots,a_{1,n-1})$ holds for every $(a_{1,1},\dots,a_{1,n-1})\in A_1$.
        \end{itemize}
    \end{enumerate}
\end{remark}

Next, we define the $n$-independence property for continuous formulas. An equivalent definition was first formulated in \cite{Chernikov2020HypergraphRA} using the $VC_n$ dimension.

\begin{defin}[$n$-independent formula]
     We say that a formula $f(x,y_0,\dots,y_{n-1})$ \emph{ has the $n$-independence property}, $IP_n$ for short, if there exist $r<s\in\mathbb{R}$ and a sequence $(a_{0,i},\dots, a_{n-1,i})_{i< \omega}$ such that for every finite $w\subseteq \omega^n$ there exists $b_w$ such that 
     \begin{align*}
         f(b_w, a_{0,i_0},\dots, a_{n-1,i_{n-1}}  )\leq r &\iff (i_0,\dots, i_{n-1})\in w\\
         &and\\
         f(b_w, a_{0,i_0},\dots, a_{n-1,i_{n-1}}  )\geq s &\iff (i_0,\dots, i_{n-1})\notin w.
     \end{align*} We say that the $\mL$-theory $T$ is \emph{$n$-dependent}, or $NIP_n$, if no $\mL$-formula has $IP_n$. 
\end{defin}

The following remark is a collection of basic properties of $n$-dependent formulas.

\begin{remark} \label{naming parameters and dummy variables}
    \begin{enumerate}
        \item Naming parameters preserves $n$-dependence. If the $\mL(A)$-formula $f(x,y_0,\dots,y_{n-1},A)$ has $IP_n$ witnessed by $(a_{0,i},\dots, a_{n-1,i})_{i< \omega}$, then the $\mL$-formula $g(x,z_0,\dots,z_{n-1})$ has $IP_n$ witnessed by $(a_{0,i}A,\dots, a_{n-1,i}A)_{i< \omega}$ where $z_i=y_iw$ and $g(x,z_0,\dots,z_{n-1})=f(x,y_0,\dots,y_{n-1},w)$.
        \item Adding dummy variables preserves $n$-dependence. Namely, let $x\subset w$ and $y_i\subset z_i$ for all $i<n$. If $g(x,z_0,\dots,z_{n-1}):=f(x,y_0,\dots,y_{n-1})$ has $IP_n$, then so does $f(x,y_0,\dots,y_{n-1})$.
        \item Every $n$-dependent theory is $(n+1)$-dependent.
    \end{enumerate}
\end{remark}

Next, we define what it means for a continuous logic formula to encode a partite and a nonpartite hypergraph.

\begin{defin}[Encoding partite hypergraphs]
    We say that a formula $f(x_0,\dots,x_{n-1})$ \emph{encodes an $n$-partite $n$-uniform hypergraph} $(G,R,P_0,\dots,P_{n-1})$ if there is a $G$-indexed sequence $(a_g)_{g\in G}$ and $r<s\in\mathbb{R}$ satisfying 
    \begin{align*}
    f( a_{g_0},\dots, a_{g_{n-1}}  )\leq r &\iff  R(g_0,\dots, g_{n-1})\\
    &and\\
    f(a_{g_0},\dots, a_{g_{n-1}}  )\geq s &\iff \neg R(g_0,\dots, g_{n-1})
    \end{align*}
    for all $g_0,\dots,g_{n-1}\in P_0\times\cdots\times P_{n-1}$.
    We say that a formula $f(x_0,\dots,x_{n-1})$ \emph{ encodes $n$-partite $n$-uniform hypergraphs} if there exist $r<s\in\mathbb{R}$ such that $f(x_0,\dots,x_{n-1})$ encodes every finite $n$-partite $n$-uniform hypergraph using the same $r$ and $s$.
\end{defin}

\begin{defin}[Encoding hypergraphs] \label{def: encoding nonpartite}
    We say that a formula $f(x_0,\dots,x_{n-1})$ \emph{encodes an $n$-uniform hypergraph $(G,R)$} if there is a $G$-indexed sequence $(a_g)_{g\in G}$  and $r<s\in\mathbb{R}$ satisfying 
    \begin{align*}
        f( a_{g_0},\dots, a_{g_{n-1}}  )\leq r &\iff  R(g_0,\dots, g_{n-1})\\
        &and\\
        f(a_{g_0},\dots, a_{g_{n-1}}  )\geq s &\iff \neg R(g_0,\dots, g_{n-1})
    \end{align*}
    for all 
    $g_0,\dots,g_{n-1}\in G$.
    We say that a formula $f(x_0,\dots,x_{n-1})$ \emph{encodes $n$-uniform hypergraphs} if there exist $r<s\in\mathbb{R}$ such that $f(x_0,\dots,x_{n-1})$ encodes every finite $n$-uniform hypergraph using the same $r$ and $s$.
\end{defin}

Note that if a formula encodes $n$-uniform hypergraphs, then it encodes $n$-partite $n$-uniform hypergraphs. 

It is no surprise that the $n$-independence property and encoding $(n+1)$-partite $(n+1)$-uniform hypergraphs are also equivalent in our continuous context.

\begin{prop} \label{IP_n iff encoding partite} Let $f(x,y_0,\dots,y_{n-1})$ be a formula. Then, the following are equivalent:
\begin{enumerate}
    \item $f$ has $IP_n$.
    \item $f$ encodes $(n+1)$-partite $(n+1)$-uniform hypergraphs.
    \item $f$ encodes $G_{n+1,p}$ as a partite hypergraph.
    \item $f$ encodes $G_{n+1,p}$ as a partite hypergraph witnessed by a $G_{n+1,p}$-indiscernible sequence.
\end{enumerate}
\end{prop}
\begin{proof}
    $(1)\implies(2)$. Let $r<s\in \mathbb{R}$, $( a_{0,i},\dots, a_{n-1,i} )_{i<\omega}$ and $(b_w)_{w\subseteq\omega^n}$ witness that $f(x,y_0,\dots,y_{n-1})$ has $IP_n$. Let a finite $n+1$-uniform $n+1$-partite hypergraph $G$ be given. Without loss of generality, we may assume that $\lvert P_0(G)\rvert=\cdots=\lvert P_n(G)\rvert=k$. For every $g\in P_0(G)$ consider the set $w_g:=\{(g_1,\dots,g_n): G\models R(g,g_1,\dots,g_n)\}$. By identifying $P_m(G)$ with $\{ (m,i): i<k\}$ and the definition of $IP_n$, we can find $b_{w_g}$ satisfying $$ f(b_{w_g}, a_{g_1},\dots, a_{g_n}  )\leq r \iff (g_1,\dots, g_{n})\in w_g$$ and $$f(b_{w_g}, a_{g_1},\dots, a_{g_n}  )\geq s \iff (g_1,\dots, g_{n})\notin w_g.$$ Then, $(a_g)_{g\in G}$ witnesses that $f$ encodes $G$, where $a_g:=b_{w_g}$ for every $g\in P_0(G)$.
    
    $(2)\implies(3)$. Follows from compactness.
    
    $(3)\implies(4)$. Let $\II=(a_g)_{g\in G_{n+1,p}}$ witness that $f(x,y_0,\dots,y_{n-1})$ encodes $G_{n+1,p}$ as a partite hypergraph. By Theorem \ref{CMP iff Ramsey} and Fact \ref{fact: ordered hypergraphs are ramsey}, there exists a $G_{n+1,p}$-indiscernible sequence $(b_g)_{g\in G_{n+1,p}}$ locally based on $\II$. It is easy to see that $(b_g)_{g\in G_{n+1,p}}$ also witnesses that $f(x,y_0,\dots,y_{n-1})$ encodes $G_{n+1,p}$ as a partite hypergraph.
    
    $(4)\implies(1)$. Let $(a_g)_{g\in G_{n+1,p}}$ witness that $f(x,y_0,\dots,y_{n-1})$ encodes $G_{n+1,p}$ as a partite hypergraph. We write $$G_{n+1,p}=\{ (j,m): j\leq n; m<\omega \}.$$ Then, by randomness of $G_{n+1,p}$ and compactness, for any  finite $w\subseteq \omega^n$ we can find $b_{w}$ such that
    \begin{align*}
        f(b_{w}, a_{1,i_1},\dots, a_{n,i_{n}}  )\leq r \iff (i_1,\dots, i_{n})\in w; \\
        f(b_{w}, a_{1,i_1},\dots, a_{n,i_{n}}  )\geq s \iff (i_1,\dots, i_{n})\notin w.
    \end{align*}
\end{proof}

As in \cite[Corollary 5.3]{Chernikov2014OnN}, from the fact that any permutation of the parts of the partition of $G_{n+1,p}$ is induced by an automorphism of $G_{n+1.p}$ treated as a pure hypergraph, we obtain the following as an easy corollary:

\begin{cor}\label{Corolary: preservation under permutation of variables}
    Let $f(x,y_0,\dots,y_{n-1})$ be a formula and $(w,z_0,\dots,z_{n-1} )$ be any permutation of $(x,y_0,\dots,y_{n-1})$. Then $g(w,z_0,\dots,z_{n-1} ):= f(x,y_0,\dots,y_{n-1})$ is $n$-dependent if and only if $f(x,y_0,\dots,y_{n-1})$ is $n$-dependent.
\end{cor}

A more involved proof is given in \cite[Proposition 10.6]{Chernikov2020HypergraphRA}.

We cannot guarantee that an $n$-independent formula will encode $(n+1)$-uniform hypergraphs. However, it is true that for continuous theories having $IP_n$ and the existence of a formula encoding $(n+1)$-uniform hypergraphs are equivalent. This generalizes the result in \cite[Lemma 2.2]{LASKOWSKI2003263} and allows is to give an alternative proof of \cite[Theorem 5.4]{Chernikov2014OnN} that avoids the mistake mentioned in the introduction. In the proof of the next result, we will write $f(y_0,\dots,y_{n-1},x)$ instead of $f(x,y_0,\dots,y_{n-1})$ for convenience.

\begin{prop}\label{IP_n iff coding nonpartite} Let $T$ be a continuous logic theory. The following are equivalent:
\begin{enumerate}
    \item T has $IP_n$.
    \item There is a continuous logic formula encoding $(n+1)$-uniform hypergraphs.
    \item There is a continuous logic formula encoding $G_{n+1}$ as a hypergraph.
    \item There is a continuous logic formula encoding $G_{n+1}$ as a hypergraph witnessed by a $G_{n+1}$-indiscernible sequence.
\end{enumerate}
\end{prop}
    \begin{proof}
            $(1)\implies(2)$. Let $f(y_0,\dots,y_{n-1},x)$ be a formula with $IP_n$. We show that the symmetric formula $$ \psi(y^0_0y^1_0\cdots y^{n-1}_0x_0,\dots,y^0_ny^1_n\cdots y^{n-1}_nx_n)=\min_{\sigma\in Sym(n)} \{ f(y^0_{\sigma(0)},\dots ,y^{n-1}_{\sigma({n-1})},x_{\sigma(n)}) \}$$ encodes every finite $(n+1)$-uniform hypergraph. By Proposition \ref{IP_n iff encoding partite}, $f$ encodes $G_{n+1,p}$ as a partite hypergraph, which is witnessed by a $G_{n+1,p}$-indiscernible sequence $(a_g)_{g\in G_{n+1,p}}$ and some $r<s\in \mathbb{R}$. We enumerate the elements of $G_{n+1,p}$ as $$\{ g^i_m: i\leq n ; m<\omega \},$$ where the superscript indicates which part of the partition the element belongs to.

            Let an $(n+1)$-uniform finite hypergraph $\mathcal{H}=(H,R_H)$ be given, we write $H:=\{ h_i: i<k \}$ for some $k<\omega$. We construct $\Tilde{\mathcal{H}}=(\Tilde{H},R_{\Tilde{H}})$ an isomorphic copy of $\mathcal{H}$ consisting of elements $\Tilde{h}_i$ for $i<k$ of the form $( g^0_i,\dots,g^{n-1}_i,g^{n}_{c(i)} )$ and show that the formula $\psi$ encodes $\Tilde{\mathcal{H}}$ (and hence encodes $\mathcal{H}$), where the relation $R_{\Tilde{H}}$ and the function $c:\omega\to\omega$ are to be defined.

            We start by defining the function $c$. For every $i<\omega$, let $c(i)$  be the smallest $m<\omega$ such that for any $(j_0,\dots,j_{n-1})\in [k]^n$ $$ R_{G_{n+1,p}}(g^0_{j_0},\dots, g^{n-1}_{j_{n-1}}, g^n_m) \iff j_0<\cdots < j_{n-1}<i \wedge R_H(h_{j_0},\dots,h_{j_{n-1}},h_i).$$ Note that the existence of such $m$ is guaranteed by randomness of $G_{n+1,p}$.

            The relation $R_{\Tilde{H}}$ is defined in the following manner: for $i_0<\cdots<i_n<k$ we set $$R_{\Tilde{H}}(\Tilde{h}_{i_0},\dots,\Tilde{h}_{i_n}) \iff R_{G_{n+1,p}}( g^0_{i_0}, g^1_{i_1},\dots,g^{n-1}_{i_{n-1}},g^n_{c(i_n)} ),$$ the rest of the cases are defined by symmetry of $R_{\Tilde{H}}$ and by declaring $\neg R_{\Tilde{H}}(\Tilde{h}_{i_0},\dots,\Tilde{h}_{i_n})$ whenever $i_{m_1}=i_{m_2}$ for some $m_1\neq m_2$. Note that by construction, we have $(H,R_{H})\cong(\Tilde{H},R_{\Tilde{H}})$.

            \begin{claim}
                The elements $b_{\Tilde{h}_i}:=(a_{g^0_{i}},\dots, a_{g^{n-1}_i},a_{ g^n_{c(i)} })$ for $i<k$ witness that $\psi$ encodes $\Tilde{\mathcal{H}}$ with the above $r<s$.
            \end{claim}
            \begin{proof}[Proof of claim]
                To ease the notation, for each $\sigma \in Sym(n)$ we write $$f_\sigma:=f(y^0_{\sigma(0)},\dots ,y^{n-1}_{\sigma({n-1})},x_{\sigma(n)}).$$ Our goal is the following: \begin{align*}
                    \psi( b_{\Tilde{h}_{i_0}},\dots b_{\Tilde{h}_{i_n}})\leq r &\iff R_{\Tilde{H}}(\Tilde{h}_{i_0},\dots,\Tilde{h}_{i_n})\\
                    \psi( b_{\Tilde{h}_{i_0}},\dots b_{\Tilde{h}_{i_n}})\geq s &\iff \neg R_{\Tilde{H}}(\Tilde{h}_{i_0},\dots,\Tilde{h}_{i_n})
                \end{align*}

                First, note that if $i_{m_1}=i_{m_2}$ for some $m_1,m_2<n$ then we have $\neg R_{\Tilde{H}}(\Tilde{h}_{i_0},\dots,\Tilde{h}_{i_n})$ by definition of $R_{\Tilde{H}}$ and $\psi( b_{\Tilde{h}_{i_0}},\dots b_{\Tilde{h}_{i_n}})\geq s$ by construction of the function $c$. Hence, we only need to prove the equivalences above in the case where all the $i$'s are pairwise distinct. We prove the first equivalence; the second one is easily deduced from it.
                
                Assume that $R_{\Tilde{H}}(\Tilde{h}_{i_0},\dots,\Tilde{h}_{i_n})$ holds. By symmetry, without loss of generality we may assume that $i_0<\cdots<i_n$. Then, by the definition of $R_{\Tilde{H}}$, this implies that $R_{ G_{n+1,p} }( g^0_{i_0},\dots, g^{n-1}_{i_{n-1}}, g^n_{c(i_n)}  )$ holds. Since the formula $f$ encodes $G_{n+1,p}$, this is equivalent to $f( a_{g^0_{i_0}},\dots, a_{g^{n-1}_{i_{n-1}}},a_{ g^n_{c(i_n)} } )\leq r$ . Hence, $\psi( b_{\Tilde{h}_{i_0}},\dots b_{\Tilde{h}_{i_n}})\leq r$.
                
                Assume $\psi( b_{\Tilde{h}_{i_0}},\dots b_{\Tilde{h}_{i_n}})\leq r$, again by symmetry, we may assume without loss of generality that $i_0<\cdots<i_n$. This implies that for some $\sigma\in Sym( n )$ we have $f_\sigma\leq r$. However, by construction of the function $c$, the only possibility is that $f_{Id}\leq r$, that is, $f( a_{g^0_{i_0}},\dots, a_{g^{n-1}_{i_{n-1}}},a_{ g^n_{c(i_n)} } )\leq r$.
                Since the formula $f$ encodes $G_{n+1,p}$, this is equivalent to $R_{ G_{n+1,p} }( g^0_{i_0},\dots, g^{n-1}_{i_{n-1}}, g^n_{c(i_n)}  )$, which implies $R_{\Tilde{H}}(\Tilde{h}_{i_0},\dots,\Tilde{h}_{i_n})$ by definition of the relation $R_{\Tilde{H}}$.
            \end{proof}

            $(2)\implies(3)$. Follows from compactness.
            
            $(3)\implies(4)$. Let $\II=(a_g)_{g\in G_{n+1}}$ witness that $\psi(x_0,x_1,\dots,x_n)$ encodes $G_{n+1}$ as a hypergraph. By Theorem \ref{CMP iff Ramsey} and Fact \ref{fact: ordered hypergraphs are ramsey}, there exists a $G_{n+1,p}$-indiscernible sequence $(b_g)_{g\in G_{n+1}}$ locally based on $\II$. It is easy to see that $(b_g)_{g\in G_{n+1}}$ also witnesses that $\psi(x_0,x_1,\dots,x_n)$ encodes $G_{n+1}$ as a hypergraph.
            
            $(4)\implies(1)$. If a formula encodes $G_{n+1}$ as a hypergraph, then it also encodes $G_{n+1,p}$ as a partite hypergraph, which implies that the formula has $IP_n$ by Proposition \ref{IP_n iff encoding partite}.
    \end{proof}

To prove the main theorem of this section we need the next two facts about hypergraphs from \cite{Chernikov2014OnN}.

Let $(G_*,\mL_{o*},\mL_{g*})$ be either $(G_n,\{ <\}, \{<, R\})$ or $(G_{n,p},\mL^n_{op},\mL^n_{opg})$. 

Let $V\subset G_*$ be a finite set and $g_0,\dots,g_{n-1}, g'_0,\dots,g'_{n-1}\in G_*\setminus V$ such that $$R(g_0,\dots,g_{n-1})\notiff R(g'_0,\dots,g'_{n-1}).$$
By $W=g_0\dots g_{n-1}V$ we mean the set $\{ g_0,\dots,g_{n-1} \}\cup V$ with the inherited structure from $G_*$. Let $\mL_*$ be $\mL_{o*}$ or $\mL_{g*}$, and let $W=g_0\dots g_{n-1}V$, $W'=g'_0\dots g'_{n-1}V$, we write $W\cong_{\mL_*}W'$ if the map acting as the identity in $V$ and sending $g_i$ to $g'_i$ for $i<n$ is an $\mL_*$ isomorphism. 

\begin{defin} Let $V\subset G_*$ be a finite set and $g_0,\dots,g_{n-1}, g'_0,\dots,g'_{n-1}\in G_*\setminus V$ be as above. 
$W=g_0\dots g_{n-1}V$ is $V$-adjacent to $W'=g'_0\dots g'_{n-1}V$ if \begin{itemize}
    \item $W\cong_{L_{o*}} W'$,
    \item for every nonempty $\overline{v}\in V$ with $\lvert \overline{v}\rvert=k$ and $i_0,\dots,i_{n-k-1}<n$ $$ R(g_{i_0},\dots g_{i_{n-k-1}},\overline{v}) \iff  R(g'_{i_0},\dots g'_{i_{n-k-1}},\overline{v})$$
\end{itemize}
$W$ is said to be adjacent to $W'$ if there is $V\subset W \cap W'$ such that $W$ is $V$-adjacent to $W'$. 
\end{defin}
\begin{fact} \label{ Sequence of adjacent graphs}
    Let $W,W'\subset G_*$ be subsets such that $W\cong_{L_{o*}} W'$. Then there is a sequence $W=W_0,W_1,\dots,W_k$ such that $W_{i+1}$ is adjacent to $W_i$ for every $i<k$ and $W_k\cong_{L_{g*}}W'$
\end{fact}
\begin{fact} \label{Isomorphic copy of random partite hypergraph}
    Let $V\subset G_*$ be a finite set and $g_0<\dots<g_{n-1} \in G_* \setminus V$ with $R(g_0,\dots,g_{n-1})$. Then there are infinite sets $X_0<\dots<X_{n-1}\subseteq G_*$ such that
    \begin{itemize}
        \item $(G';<;R)\cong (G_{n,p};<;R)$ where $G'=X_0\dots X_{n-1}$ (i.e. each $X_i$ correspond to the part $P_i$ of the partition),
        \item for any $g'_i\in X_i$ ($i<n$), either $W\cong_{\mL_{opg}} W'$ or $W$ is $V$-adjacent to $W'$, where $W=g_0\dots g_{n-1}V$ and $W'=g'_0\dots g'_{n-1}V$
    \end{itemize}
\end{fact}

We are now ready to prove the main theorem of the section.

\begin{teor} \label{n-dep and collapsing}
    Let $T$ be a complete continuous logic theory. The following are equivalent:
    \begin{enumerate}
        \item $T$ is $n$-dependent.
        \item Every $G_{n+1,p}$-indiscernible is $\mL_{op}$-indiscernible.
        \item Every $G_{n+1}$-indiscernible is order-indiscernible.
    \end{enumerate}
\end{teor}
    \begin{proof}
            $(1)\implies(2)$. Let $(a_g)_{g\in G_{n+1,p}}$ be a $G_{n+1,p}$-indiscernible sequence which is not $\mL_{op}$-indiscernible. Then there are $\mL_{op}$-isomorphic $W,W'\subset G_{n+1}$ subsets of size $m$, a formula $f(x_0,\dots,x_{m-1})$ and $r<s\in \mathbb{R}$ such that $f( (a_g)_{g\in W}) \leq r$ and $f( (a_g)_{g\in W'})\geq s$ (where the elements $a_g$ are substituted for the variables $x_0,\dots,x_{m-1}$ according to the ordering on $W$ and $W'$). Without loss of generality, by Fact \ref{ Sequence of adjacent graphs}, we may assume that $W$ is $V$-adjacent to $W'$ for some subset $V$ such that $W=g_0\dots g_n V$, $W'=g'_0\dots g'_n V$, $R(g_0,\dots g_n)$ and $\neg R(g'_0,\dots g'_n)$. 

            Now we apply Fact \ref{Isomorphic copy of random partite hypergraph} to $V$ and $g_0,\dots,g_n$. This yields $G'\subseteq G_{n+1,p}$ such that for every $(h_0,\dots,h_n)\in \prod_{i\leq n} P_i(G')$ $$ R(h_{0},\dots,h_{n}) \iff h_0\dots h_n V\cong_{\mL_{opg}}W $$ and $$ \neg R(h_0,\dots,h_n) \iff h_0\dots h_nV\cong_{\mL_{opg}}W'.$$ Recall that the sequence $(a_g)_{g\in G_{n+1,p}}$ is $G_{n+1,p}$-indiscernible and let $f'$ be the formula defined by permuting the variables $(x_0,\dots,x_{m-1})$ of $f$ in such a way that the first $n+1$-variables are the ones corresponding to $h_0,\dots,h_n$ according to the ordering on the set $\{h_0,\dots,h_n\}\cup V$. Then,
            $$ f'(a_{h_0},\dots, a_{h_n},A)\leq r \iff R(h_{0},\dots,h_{n})$$
            and
            $$f'(a_{h_0},\dots, a_{h_n},A)\geq s \iff \neg R(h_{0},\dots,h_{n}),$$ where $A=(a_g)_{g\in V}$. Since $G'$ is isomorphic to $G_{n+1,p}$, by Proposition \ref{IP_n iff encoding partite}, the formula $f'(x,y_{0},\dots, y_{n-1},A)$ has $IP_n$ and hence, by Remark \ref{naming parameters and dummy variables} there is a continuous logic formula $g(x,z_{0},\dots, z_{n-1})$ which has $IP_n$.
            
            $(1)\implies(3)$. The proof is exactly as the proof of $(1)\implies (2)$.
            
            $(2)\implies(1)$. It follows from Proposition \ref{IP_n iff encoding partite}. If the formula $f$ encodes $G_{n+1,p}$ as a partite hypergraph witnessed by a $G_{n+1,p}$-indiscernible sequence $(a_g)_{g\in G_{n+1,p}}$, then $(a_g)_{g\in G_{n+1,p}}$ cannot be $\mL_{op}$-indiscernible.
            
            $(3)\implies(1)$ Follows from Proposition \ref{IP_n iff coding nonpartite}. If $T$ has $IP_n$, there is a continuous logic formula encoding $G_{n+1}$ as a hypergraph witnessed by a $G_{n+1}$-indiscernible sequence$ (a_g)_{g\in G_{n+1}}$. Then $(a_g)_{g\in G_{n+1}}$ cannot be order-indiscernible.
    \end{proof}

We know explain the error in the proof of \cite[Theorem 5.4 $(3)\implies (2)$]{Chernikov2014OnN}

Without loss of generality we write $$G_{n+1,p}=\{ g^i_q: i<n; q\in \mathbb{Q}\},$$ where $g^i_q \in P_i(G_{n+1,p})$ and $g^i_q<g^i_p$ for all $q<p\in \mathbb{Q}$. We define the ordered $(n+1)$-uniform hypergraph $G^*_{n+1}$ as follows:
\begin{itemize}
    \item $G^*_{n+1}=\{ h_q: h_q=(g^0_q,\dots, g^n_q), q\in \mathbb{Q} \}$,
    \item $ R_{G^*_{n+1}}( h_{q_0},\dots, h_{q_n} ) \iff R_{G_{n+1,p}}( g^0_{q_0},\dots, g^n_{q_n} ) $ for $q_0<\dots<q_n$,
    \item $h_q<h_p \iff q<p$.  
\end{itemize}
Clearly, $G^*_{n+1}$ embeds every finite ordered $(n+1)$-uniform hypergraph.

Let $(a_g)_{g\in G_{n+1,p}}$ be a $G_{n+1,p}$-indiscernible sequence which is not $\mL_{op}$-indiscernible. For $h_q\in G^*_{n+1}$, let $b_{h_q}=( a_{g^0_q},\dots, a_{g^n_q} )$ and consider the $G^*_{n+1}$-indexed sequence $(b_h)_{h\in G^*_{n+1}}$. The following claim is made in the proof:
\begin{claim}
    Whenever $X\equiv_{<_{G^*_{n+1}},R_{G^*_{n+1}}}Y \subseteq G^*_{n+1}$, we have $$ \tp( (b_h)_{h\in X} )=\tp((b_h)_{h\in Y}) .$$
\end{claim}
However, this claim is not true as shown by the following counterexample. 
\begin{counterexample}\label{counterexample}
    Consider the theory $T=\Th(G_{n+1,p})$ and the sequence $(g)_{g\in G_{n+1,p}}$ (This sequence corresponds to the sequence $(a_g)_{g\in G_{n+1,p}}$ of the claim, which is clearly $G_{n+1,p}$-indiscernible but not order-indiscernible). Then, for $h_q=(g^0_q,\dots, g^n_q)$ we have $b_{h_q}:=(g^0_q,\dots,g^n_q)$. Let $X:=\{h_{q_0},h_{q_1} \}$ for some $q_0<q_1\in \mathbb{Q}$, by randomness, there is $q_0<\Tilde{q}\in \mathbb{Q}$ and $h_{\Tilde{q}}=(g^0_{\Tilde{q}},\dots,g^n_{\Tilde{q}})$ such that $$ R_{ G_{n+1,p}}(g^0_{q_0},\dots, ,g^{n-1}_{q_0}, g^n_{q_1})\iff \neg  R_{ G_{n+1,p}}(g^0_{q_0},g^1_{q_0},\dots,g^{n-1}_{q_0}, g^n_{\Tilde{q}}).$$
Thus, $h_{q_0},h_{q_1}\equiv_{<_{G^*_{n+1}},R_{G^*_{n+1}}}  h_{q_0},h_{\Tilde{q}}$ since: \begin{itemize}
    \item $\Th(G^*_{n+1})$ has quantifier elimination,
    \item $h_{q_0}<h_{q_1}$ and $h_{q_0}<h_{\Tilde{q}}$,
    \item $R_{G^*_{n+1}}$ has arity $n+1$.
\end{itemize}

We also have $\tp(b_{h_{q_0}},b_{h_{q_1}})\neq \tp(b_{h_{q_0}},b_{h_{\Tilde{q}}})$ since, by our choice of $\Tilde{q}$, $$ R_{ G_{n+1,p}}(g^0_{q_0},\dots, ,g^{n-1}_{q_0}, g^n_{q_1})\iff \neg  R_{ G_{n+1,p}}(g^0_{q_0},g^1_{q_0},\dots,g^{n-1}_{q_0}, g^n_{\Tilde{q}}).$$
\end{counterexample}

The following is due to Artem Chernikov, Daniel Palacín and Kota Takeuchi. If we substitute the claim above in the original proof by the following: Let $A\subset G_{n+1,p}$ be any finite substructure. Without loss of generality, we may assume that if $g^i_q,g^j_p\in A$ and $i<j$ then $q<p$. Let $A^*=\{ h_q: g^i_q\in A \}$ and let $\varphi^*((x^0_q,\dots,x^{n-1}_q)_{h_q\in A^*}):=\varphi((x^i_q)_{g^i_q\in A})$ for each formula $\varphi((x^i_q)_{g^i_q\in A})$. 
\begin{claim}
  Whenever $A^*\equiv_{<_{G^*_{n+1}},R_{G^*_{n+1}}}X \subseteq G^*_{n+1}$, we have $$ \tp_{\varphi^*}( (b_h)_{h\in A^*} )=\tp_{\varphi^*}((b_h)_{h\in X}) .$$
\end{claim}
Then, the proof goes through. Our counterexample shows that one has to work with a restricted family of formulas $\varphi^*$.


The formula $\psi$ that we constructed in the proof of Proposition \ref{IP_n iff coding nonpartite} will be important throughout the remainder of the chapter.

\begin{notation}\label{notation psi_f}
     Given a continuous logic formula $f(y_0,\dots,y_{n-1},x)$ we denote the symmetric formula constructed in the proof of Proposition \ref{IP_n iff coding nonpartite} as $\psi_f$. Namely: $$ \psi_f(y^0_0y^1_0\cdots y^{n-1}_0x_0,\dots,y^0_ny^1_n\cdots y^{n-1}_nx_n)=\min_{\sigma\in Sym(n)} \{ f(y^0_{\sigma(0)},\dots ,y^{n-1}_{\sigma({n-1})},x_{\sigma(n)}) \}.$$
\end{notation}

It turns out that $IP_n$ of the formulas $f$ and $\psi_f$ are equivalent. The implication $(2)\implies (1)$ in the next result can also be deduced from \cite[Proposition 10.4 and Proposition 10.6]{Chernikov2020HypergraphRA} since the set of connectives $\{\neg, \frac{1}{2}, \dot{-}\}$ is full (we can approximate every continuous formula uniformly by formulas constructed using only that set of connectives). However, we provide a direct proof with ideas that will be useful in the next section.
\begin{lema} \label{IP_n f and psi}
    Let $f(y_0,\dots,y_{n-1},x)$ be a continuous logic formula. Then, the following are equivalent:
    \begin{enumerate}
        \item $f$ has $IP_n$.
        \item $\psi_f$ has $IP_n$.
    \end{enumerate}
\end{lema}
\begin{proof}
        $(1)\implies (2)$. Follows from the proof of $(1)\implies (2)$ of Proposition \ref{IP_n iff coding nonpartite} and Proposition \ref{IP_n iff encoding partite}.
        
        $(2)\implies (1)$. First, recall that if a formula is $n$-dependent and we add dummy variables, then the new formula is still $n$-dependent and that $n$-dependence is preserved under permutations of variables (by Remark \ref{naming parameters and dummy variables} and Corollary \ref{Corolary: preservation under permutation of variables}). 
    
    We consider the formula $$f_\sigma(y^0_0y^1_0\cdots y^{n-1}_0x_0,\dots,y^0_ny^1_n\cdots y^{n-1}_nx_n)=f(y^0_{\sigma(0)},\dots ,y^{n-1}_{\sigma({n-1})},x_{\sigma(n)}).$$We have $\psi_f=\min_{\sigma\in Sym(n)} \{f_\sigma\}$.
    Let $(b_g)_{g\in G_{n+1,p}}$ be a witness that $\psi_f$ encodes $G_{n+1,p}$ as a partite hypergraph with some $r<s$, which exists by Proposition \ref{IP_n iff encoding partite} and the assumption that $\psi_f$ has $IP_n$. Fix a linear ordering of $Sym(n)$ and color the edges of $G_{n+1,p}$ according to the first $\sigma$ such that $f_\sigma(b_{g_0},\dots,b_{g_n})\leq r$ whenever $\psi_f(b_{g_0},\dots,b_{g_n})\leq r$. Let $H$ be any finite ordered $(n+1)$-partite $(n+1)$-uniform hypergraph. Since $\age(G_{n+1,p})$ has ERP, we can find a monochromatic isomorphic copy of $H$ inside $G_{n+1,p}$. This implies that $f_\sigma$ encodes $H$ as a partite hypergraph with the same $r<s$ as above (where $\sigma$ is the color of the edges of this monochromatic copy of $H$). Since each finite $(n+1)$-partite $(n+1)$-uniform hypergraph is encoded by some $f_\sigma$, by compactness there is $f_\sigma$ encoding $G_{n+1,p}$ as a partite hypergraph. Thus, $f_\sigma$ has $IP_n$ which, by the previous paragraph, implies that $f(y_0,\dots,y_{n-1},x)$ has $IP_n$.
\end{proof}

\section{$n$-dependence for hyperdefinable sets}\label{section: hyperdef n-dep}

Let $T$ be a complete, first-order theory, and $\C \models T$ a monster model (i.e. $\kappa$-saturated and strongly $\kappa$-homogeneous for a strong limit cardinal $> |T|$). Let $E$ be a $\emptyset$-type-definable equivalence relation on a $\emptyset$-type-definable subset $X$ of $\C^\lambda$ (or a product of sorts), where $\lambda< \kappa$.


We recall the family $\mathcal{F}_{X/E}$ defined in \cite[Section 2]{10.1215/00294527-2022-0023}. $\mathcal{F}_{X/E}$ is the family  of all functions $f : X \times \mathfrak{C}^m \to \mathbb{R}$ which factor through $X/E\times\mathfrak{C}^m$ and can be extended to a continuous logic formula $\C^{\lambda} \times \C^m \to \mathbb{R}$ over $\emptyset$ (i.e. factors through a continuous function $S_{\lambda+m}(\emptyset))$, where $m$ ranges over $\omega$. 


Let $A \subset \C$ (be small). Recall that the complete types over $A$ of elements of $X/E$ can be defined as the $\aut(\C/A)$-orbits on $X/E$,  or the preimages of these orbits under the quotient map, or the partial types defining these preimages. The space of all such types is denoted by $S_{X/E}(A)$.

In this section we apply the results obtained in continuous logic to the context of hyperdefinable sets to obtain a counterpart to Theorem \ref{n-dep and collapsing}. 

 In 
 \cite[Proposition 2.1]{10.1215/00294527-2022-0023},
 we showed that the family of functions $\mathcal{F}_{X/E}$ separates points in $S_{X/E\times \C^m}(\emptyset)$. Namely,
\begin{fact}\label{diff type diff f value}
For any $a_1=a'_1/E$, $a_2=a'_2/E$ in $X/E$ and $b_1,b_2\in \mathfrak{C}^m$ 
		$$\tp(a_1,b_1)\neq \tp(a_2,b_2)\iff (\exists f\in \mathcal{F}_{X/E})(f(a'_1,b_1)\neq f(a'_2,b_2)) $$
\end{fact}
This allows us to work with elements of $X/E$ as real elements if we restrict ourselves to functions from the family $\mathcal{F}_{X/E}$. Hence, we introduce the following notation:

\begin{notation} Let $\Delta$ be a set of (continuous) formulas in variables $(x_i)_{i<\lambda}$ all from the same product of sorts. We say that a sequence $(a_i)_{i\in \I}$ of elements from the appropriate product of sorts is $\I$-indiscernible with respect to $\Delta$ if for any tuples $i_1,\dots,i_n,j_1,\dots,j_n\in \I$ we have that 
    $$\qftp(i_1,\dots,i_n)=\qftp(j_1,\dots,j_n) \implies \tp^\Delta(a_{i_1}, \dots, a_{i_n})=\tp^\Delta(a_{j_1}, \dots, a_{j_n}),$$
    where the tuples $a_i$ are substituted for the variables of the formulas from $\Delta$.
\end{notation}

We define generalised indiscernible sequences of hyperimaginaries exactly as we did in Definition \ref{defin: gen. indisc.}.
\begin{defin}
    Let $\II=(a_i: i\in \I)$ be an $\I$-indexed sequence of hyperimaginaries (maybe of different sorts), and let $A\subset \C$ be a small set of parameters.  We say that $\II$ is an $\I$-\emph{indexed indiscernible sequence over} $A$ if for all $n\in\omega$ and all sequences $i_1,\dots,i_n,j_1,\dots,j_n$ from $\I$ we have that 
    $$\qftp(i_1,\dots,i_n)=\qftp(j_1,\dots,j_n) \implies \tp({a}_{i_1}, \dots, {a}_{i_n}/A)=\tp({a}_{j_1}, \dots, {a}_{j_n}/A).$$
\end{defin}

By Fact \ref{diff type diff f value}, a sequence of hyperimaginaries $(a_i/E)_{i\in \I}$ is $\I$-indiscernible if and only if the sequence $(a_i)_{i\in \I}$ is $\I$-indiscernible with respect to the family of functions $f: X^n\to \R$ that factor through $(X/E)^n$ and can be extended to a continuous formula $f:\C^{n\lambda}\to \R$ over $\emptyset$ where $n$ ranges over $\omega$.

Next, we define the $n$-independence property for hyperdefinable sets. 

\begin{defin}\label{Definition: hyperim IP_n}
A hyperdefinable set $X/E$ has the \emph{$n$-independence property}, $IP_n$ for short, if for some $m<\omega$ there exist two distinct complete types $p,q\in S_{X/E\times \C^m}(\emptyset)$ and a sequence $(a_{0,i},\dots, a_{n-1,i})_{i< \omega}$ such that for every finite $w\subset \omega^n$ there exists $b_w\in X/E$ satisfying $$\tp(b_w, a_{0,i_0},\dots, a_{n-1,i_{n-1}}  )=p \iff (i_0,\dots, i_{n-1})\in w$$ $$\tp(b_w, a_{0,i_0},\dots, a_{n-1,i_{n-1}}  )=q \iff (i_0,\dots, i_{n-1})\notin w. $$ 
\end{defin}

Note that $1$-dependent hyperdefinable sets are exactly the hyperdefinable sets with $NIP$ (see the definition of hyperdefinable set with NIP in \cite[Remark 2.3]{HaskelPilllay}) by 
\cite[Lemma 2.12]{10.1215/00294527-2022-0023}.

We can easily modify our definition of $n$-independence to suit functions in $\mathcal{F}_{X/E}$.

\begin{defin}\label{defin: n-dep function/E}
     We say that $f(x,y_0,\dots,y_{n-1})\in \mathcal{F}_{X/E}$ \emph{ has the $n$-independence property}, $IP_n$ for short, if there exist $r<s\in\mathbb{R}$ and a sequence $(a_{0,i},\dots, a_{n-1,i})_{i< \omega}$ from anywhere such that for every finite $w\subseteq \omega^n$ there exists $b_w\in X$ satisfying 
     \begin{align*}
         f(b_w, a_{0,i_0},\dots, a_{n-1,i_{n-1}}  )\leq r &\iff (i_0,\dots, i_{n-1})\in w\\
         &and\\
         f(b_w, a_{0,i_0},\dots, a_{n-1,i_{n-1}}  )\geq s &\iff (i_0,\dots, i_{n-1})\notin w.
     \end{align*} 
\end{defin}


Similarly, we can define what it means for a function in $\mathcal{F}_{X/E}$ to encode ($n$-partite) $n$-uniform hypergraphs.
\begin{defin}
    We say that  $f(x,y_1,\dots,y_{n-1})\in \mathcal{F}_{X/E}$ \emph{encodes a $n$-partite $n$-uniform hypergraph} $(G,R,P_0,\dots,P_{n-1})$ if there are a $G$-indexed sequence $(a_g)_{g\in G}$ with $a_g\in X$ for every $g\in P_0(G)$ and $r<s\in\mathbb{R}$ satisfying 
    \begin{align*}
    f( a_{g_0},\dots, a_{g_{n-1}}  )\leq r &\iff  R(g_0,\dots, g_{n-1})\\
    &and\\
    f(a_{g_0},\dots, a_{g_{n-1}}  )\geq s &\iff \neg R(g_0,\dots, g_{n-1})
    \end{align*}
    for all $g_0,\dots,g_{n-1}\in P_0\times\cdots\times P_{n-1}$.
    We say that $f(x_0,\dots,x_{n-1})\in \mathcal{F}_{X/E}$ \emph{ encodes $n$-partite $n$-uniform hypergraphs} if there exist $r<s\in\mathbb{R}$ such that $f(x,y_1,\dots,y_{n-1})$ encodes every finite $n$-partite $n$-uniform hypergraph using the same $r$ and $s$.
\end{defin}

The proof of the following fact is exactly as in the case of a general continuous formula.

\begin{prop} \label{Hyperim: IP_n iff encoding partite} Let $f(x,y_0,\dots,y_{n-1})\in \mathcal{F}_{X/E}$. Then, the following are equivalent:
\begin{enumerate}
    \item $f$ has $IP_n$.
    \item $f$ encodes $(n+1)$-partite $(n+1)$-uniform hypergraphs.
    \item $f$ encodes $G_{n+1,p}$ as a partite hypergraph.
    \item $f$ encodes $G_{n+1,p}$ as a partite hypergraph witnessed by a $G_{n+1,p}$-indiscernible sequence.
\end{enumerate}
\end{prop}

The following lemma allows us to understand $IP_n$ of a hyperdefinable set $X/E$ through the family of functions $\mathcal{F}_{X/E}$.

\begin{lema}\label{IP_n hyperdef. set iif formulas}
    $X/E$ has $IP_n$ if and only if some $f(x,y_0,\dots,y_{n-1})\in \mathcal{F}_{X/E}$ has $IP_n$.
\end{lema}
\begin{proof}
    Assume that $X/E$ has $IP_n$. Take witnesses $p$ and $q$ from Definition \ref{Definition: hyperim IP_n}. Then, by Fact \ref{diff type diff f value}, there exists $f(x,y_0,\dots,y_{n-1})\in\mathcal{F}_{X/E}$ and $r<s$ such that $f(x,y_0,\dots,y_{n-1})\leq r \in p $ and $f(x,y_0,\dots,y_{n-1})\geq s \in q$. The elements witnessing $IP_n$ for $X/E$ also  witness that $f(x,y_0,\dots,y_{n-1})$ has $IP_n$.

 Assume now that some $f(x,y_0,\dots,y_{n-1})\in\mathcal{F}_{X/E}$ has $IP_n$. By Proposition \ref{Hyperim: IP_n iff encoding partite}, the function $f$ encodes $G_{n+1,p}$ as a partite hypergraph witnessed by a $G_{n+1,p}$-indiscernible sequence $(a_g)_{g\in G_{n+1,p}}$ and some $r<s\in \R$. We write $$G_{n+1,p}=\{ (j,m): j\leq n; m<\omega \}.$$ Then, by randomness of $G_{n+1,p}$, for any  finite disjoint $s_0,s_1\subseteq \omega^n$ we can find $b_{s_0,s_1}\in X$ (in fact, we can choose it to be some $a_g$ for $g\in P_0(G_{n+1,p})$) such that
    \begin{align*}
        (i_1,\dots, i_{n})\in s_0\implies f(b_{s_0,s_1}, a_{1,i_1},\dots, a_{n,i_{n}}  )\leq r,\\
       (i_1,\dots, i_{n})\in s_1 \implies f(b_{s_0,s_1}, a_{1,i_1},\dots, a_{n,i_{n}}  )\geq s .
    \end{align*}
    Moreover, by $G_{n+1,p}$-indiscernibility, there exist two distinct complete types $p,q\in S_{X/E\times \C^m}(\emptyset)$ such that 
    \begin{align*}
        (i_1,\dots, i_{n})\in s_0 \implies \tp(b_{s_0,s_1}/E, a_{1,i_1},\dots, a_{n,i_{n}}  )=p; \\
        (i_1,\dots, i_{n})\in s_1 \implies \tp(b_{s_0,s_1}/E, a_{1,i_1},\dots, a_{n,i_{n}}  )=q.
    \end{align*}

    By compactness, it follows that $X/E$ has $IP_n$.
\end{proof}

 \begin{notation}\label{notation PSI}
     For each $f(x,y_0,\dots, y_n)\in \mathcal{F}_{X/E}$, let $f'(y_0,\dots, y_n,x):=f(x,y_0,\dots, y_n)$. We denote by $\Psi^{n+1}_{ \mathcal{F}_{X/E} }$ the set containing all functions $\psi_{f'}$ for $f(x,y_0,\dots, y_n)\in \mathcal{F}_{X/E}$, where $\psi_{f'}$ is constructed as in Notation \ref{notation psi_f}. $\Psi_{ \mathcal{F}_{X/E} }$ is the union of all $\Psi^{n+1}_{ \mathcal{F}_{X/E} }$ for $n<\omega$.
 \end{notation}

 The next definition is the natural counterpart of Definition \ref{def: encoding nonpartite} for functions of the family $\Psi_{ \mathcal{F}_{X/E} }$.


\begin{defin}
    We say that $\psi(x_0,\dots,x_{n-1})\in \Psi_{ \mathcal{F}_{X/E} } $ \emph{encodes an $n$-uniform hypergraph $(G,R)$} if there is a $G$-indexed sequence $(a_g)_{g\in G}$ in $\C^m\times X$ (for some fixed $m<\omega$)  and $r<s\in\mathbb{R}$ satisfying 
    \begin{align*}
        \psi( a_{g_0},\dots, a_{g_{n-1}}  )\leq r &\iff  R(g_0,\dots, g_{n-1})\\
        &and\\
        \psi(a_{g_0},\dots, a_{g_{n-1}}  )\geq s &\iff \neg R(g_0,\dots, g_{n-1})
    \end{align*}
    for all 
    $g_0,\dots,g_{n-1}\in G$.
    We say that $\psi(x_0,\dots,x_{n-1})$ \emph{encodes $n$-uniform hypergraphs} if there exist $r<s\in\mathbb{R}$ such that $\psi(x_0,\dots,x_{n-1})$ encodes every finite $n$-uniform hypergraph using the same $r$ and $s$.
\end{defin}

As in the general continuous case, we have an equivalence between $IP_n$ of the hyperdefinable set $X/E$ and the existence of some function coding $(n+1)$-uniform hypergraphs.

\begin{prop}\label{Hyperim: IP_n iff coding nonpartite} 
The following are equivalent:
\begin{enumerate}
    \item $X/E$ has $IP_n$.
    \item There is a function in $\Psi_{\mathcal{F}_{X/E}}$ encoding $(n+1)$-uniform hypergraphs.
    \item There is a function in $\Psi_{\mathcal{F}_{X/E}}$ encoding $G_{n+1}$ as a hypergraph.
    \item There is a function in $\Psi_{\mathcal{F}_{X/E}}$ encoding $G_{n+1}$ as a hypergraph witnessed by a $G_{n+1}$-indiscernible sequence.
\end{enumerate}
\end{prop}

\begin{proof}
    $(1)\implies (2)$. By Lemma \ref{IP_n hyperdef. set iif formulas} and Proposition \ref{Hyperim: IP_n iff coding nonpartite}, there exists a function $f(x,y_0,\dots,y_{n-1})\in\mathcal{F}_{X/E}$ encoding $G_{n+1,p}$ as a partite hypergraph. Consider the function $f'(y_0,\dots,y_{n-1},x):=f(x,y_0,\dots,y_{n-1})$, following the proof of Proposition \ref{IP_n iff coding nonpartite}, we see that $\psi_{f'}\in \Psi_{\mathcal{F}_{X/E}}$ encodes $G_{n+1}$. 
    
    $(2)\implies (3)$ follows by compactness.
    
    $(3) \implies (4)$. Follows from the proof of $(3)\implies (4)$ of Proposition \ref{IP_n iff coding nonpartite}.
    

    $(4)\implies (1)$ We slightly modify the proof of Lemma \ref{IP_n f and psi}. Let $(b_g)_{g\in G_{n+1}}$ be a witness that $\psi_f$ encodes $G_{n+1}$ as a hypergraph. Note that for every $g\in G_{n+1}$, $b_g=(a^0_g,\dots, a^n_g)$ with $a^n_g\in X$. Fix a linear ordering of $Sym(n)$ and color the edges of $G_{n+1}$ according to the first $\sigma$ such that $f(a^0_{g_{\sigma(0)}},\dots,a^n_{g_{\sigma(n)}})\leq r$ whenever $\psi_f(b_{g_0},\dots,b_{g_n})\leq r$. Let $H=\{ h^i_m: i\leq n, m\leq k \}$ be any finite ordered $(n+1)$-partite $(n+1)$-uniform hypergraph. Since $\age(G_{n+1})$ has ERP, we can find a monochromatic isomorphic copy of $H$ inside $G_{n+1}$ (as a non partite hypergraph). This implies that the function $f(y_0,\dots,y_{n-1},x)$ encodes $H$ as a partite hypergraph, witnessed by the elements $\{ a^i_{h^{\sigma(i)}_m}: i\leq n, m\leq k \},$ where $\sigma$ is the color of the monochromatic copy of $H$. By compactness, $f'(x,y_0,\dots,y_{n-1}):=f(y_0,\dots,y_{n-1},x)$ encodes $G_{n+1,p}$ as a partite hypergraph. Therefore, since the function $f'(x,y_0,\dots,y_{n-1})$ is in $\mathcal{F}_{X/E}$, by Proposition \ref{Hyperim: IP_n iff encoding partite}, $X/E$ has $IP_n$.
\end{proof}

We finish the section with a characterization of $n$-dependent hyperdefinable sets analogous to the one in Theorem \ref{n-dep and collapsing}. Recall the definition of the family $\Psi^{n+1}_{\mathcal{F}_{X/E}}$ from Notation \ref{notation PSI}.

\begin{teor}
The following are equivalent:
    \begin{enumerate}
        \item $X/E$ is $n$-dependent.
        \item Every $G_{n+1,p}$-indiscernible $(a_g)_{g\in G_{n+1,p}}$ where for every $g\in P_0(G_{n+1,p})$ we have $a_g\in X/E$ is $\mL_{op}$-indiscernible.
        \item For every $m\in \mathbb{N}$, every $G_{n+1}$-indiscernible with respect to $\Psi^{n+1}_{\mathcal{F}_{X/E}}$ sequence of elements of $\C^m\times X$  is order-indiscernible with respect to $ \Psi^{n+1}_{\mathcal{F}_{X/E}}$.
    \end{enumerate}
\end{teor}

\begin{proof}
            $ (1)\implies (2)$. Let $(a_g)_{g\in G_{n+1,p}}$ be a $G_{n+1,p}$-indiscernible sequence where for every $g\in P_0(G_{n+1,p})$ we have $a_g\in X/E$ which is not $\mL_{op}$-indiscernible and let $(a'_g)_{g\in G_{n+1,p}}$ be a sequence of representatives. Then, by Lemma \ref{diff type diff f value}, there are $\mL_{op}$-isomorphic $W,W'\subset G_{n+1}$ subsets of size $m$, a function $f(x_0,\dots,x_{m-1})$ and $r<s\in \mathbb{R}$ such that $f( (a'_g)_{g\in W}) \leq r$ and $f( (a'_g)_{g\in W'})\geq s$ (where the elements $a'_g$ are substituted for the variables $x_0,\dots,x_{m-1}$ according to the ordering on $W$ and $W'$). Without loss of generality, by Fact \ref{ Sequence of adjacent graphs}, we may assume that $W$ is $V$-adjacent to $W'$ for some subset $V$ such that $W=g_0\dots g_n V$, $W'=g'_0\dots g'_n V$, $R(g_0,\dots g_n)$ and $\neg R(g'_0,\dots g'_n)$. 
            Following as in the proof of $(1) \implies (2)$ of Theorem \ref{n-dep and collapsing}, we arrive at the conclusion that $f'(x,y_0,\dots,y_n,A)$ has $IP_n$, where $A=(a'_g)_{g\in V}$. The tuple $A$ is contained in some product (with repetition) of $X$ and $\C$. Hence, the function $g$ obtained at the end of the proof of this implication in Theorem \ref{n-dep and collapsing} possibly has several infinite tuples of variables each of which corresponds to $X$. Thus, when performing the change of variables done in Remark \ref{naming parameters and dummy variables} (1) we might end with infinite tuples of variables. However, since this new function $g(x,z_1\dots,z_n)$ is the uniform limit of functions from the family $\mathcal{F}_{X/E}$, we might find a suitable formula $g'\in \mathcal{F}_{X/E}$ witnessing $IP_n$.
            
            $(1)\implies (3)$. 
            Let $(a_g)_{g\in G_{n+1}}$ be a $G_{n+1}$-indiscernible with respect to $\Psi^{n+1}_{\mathcal{F}_{X/E}}$ sequence which is not order-indiscernible with respect to $\Psi^{n+1}_{\mathcal{F}_{X/E}}$. Then there are $W,W'\subset G_{n+1}$ subsets of size $n+1$, a function $\psi_f(x_0,\dots,x_{n})$ and $r<s\in \mathbb{R}$ such that $\psi_f( (a_g)_{g\in W}) \leq r$ and $\psi_f( (a_g)_{g\in W'})\geq s$. By Fact \ref{ Sequence of adjacent graphs} and the fact that $G_{n+1}$ is self-complementary, 
            we may assume $W=g_0\dots g_n$, $W'=g'_0\dots g'_n$, $R(g_0,\dots g_n)$ and $\neg R(g'_0,\dots g'_n)$.

            By $G_{n+1}$-indiscernibility of $(a_g)_{g\in G_{n+1}}$ with respect to  $\Psi^{n+1}_{\mathcal{F}_{X/E}}$ and by symmetry of the relation $R$ and $\psi_f$, this implies that 
            $$ \psi_f(a_{h_0},\dots, a_{h_n})\leq r \iff R(h_{0},\dots,h_{n})$$
            and
            $$\psi_f(a_{h_0},\dots, a_{h_n})\geq s \iff \neg R(h_{0},\dots,h_{n}).$$ By Proposition \ref{Hyperim: IP_n iff coding nonpartite}, the set $X/E$ has $IP_n$.

            $(2)\implies(1)$ It follows from Proposition \ref{Hyperim: IP_n iff encoding partite}. If the function $f\in \mathcal{F}_{X/E}$ encodes $G_{n+1,p}$ witnessed by a $G_{n+1,p}$-indiscernible sequence $(a_g)_{g\in G_{n+1,p}}$, then $(a_g)_{g\in G_{n+1,p}}$ cannot be $\mL_{op}$-indiscernible.
            
            $(3)\implies(1)$ It follows from Proposition \ref{Hyperim: IP_n iff coding nonpartite}. If $T$ has $IP_n$, there is a function $\psi_f\in \Psi^{n+1}_{\mathcal{F}_{X/E}}$ encoding $G_{n+1}$ witnessed by a $G_{n+1}$-indiscernible sequence $(a_g)_{g\in G_{n+1}}$. Then, $(a_g)_{g\in G_{n+1}}$ cannot be order-indiscernible respect to $ \Psi^{n+1}_{\mathcal{F}_{X/E}}$.
\end{proof}
Note that the results of this section easily generalise to deal with $n$-dependence of imaginary sorts in continuous logic. We finish the chapter with an example illustrating that, in general, the theorem above is optimal. Namely, for an $n$-dependent hyperdefinable set $X/E$ and $n'>n$ there might be a $G_{n+1}$-indiscernible sequence of elements of $\C^m\times X$ for some $m<\omega$ which are not order-indiscernible with respect to $\Psi^{n'+1}_{\mathcal{F}_{X/E}}$ or with respect to more general families of functions from $\mathcal{F}_{(\C^m \times X/E)^{n+1}}$.
\begin{example}\label{example optimal}
    Let $\mathcal{N}$ be a monster model of a NIP theory and $\mathcal{R}$ a monster model of the theory of random ordered graphs. We consider the structure $\mathcal{N}\sqcup \mathcal{R}$ i.e. the structure with disjoint sorts for $\mathcal{N}$ and $\mathcal{R}$ with no interaction between the sorts. Let $X=\mathcal{N}$ and $E$ be the equality relation. Clearly, $X/E$ has NIP.

    \begin{claim}
        Let $n\geq 1$. The sequence $(a_g)_{g\in G_2}:=(\overbrace{g,\dots,g}^m,n_0)_{g\in G_2}$ is $G_2$-indiscernible but it is not order-indiscernible with respect to the family of formulas $f(x_0,\dots,x_n)$ where each $x_i$ is a tuple $(x_i^0,\dots,x_i^m)$ of variables of length $m+1$ whose last coordinate corresponds to $X$.
    \end{claim}

    \begin{proof}[Proof of the first claim]
    Let $g_0'<g_0<g_1<\dots<g_n\in G_2$ be such that $R(g_0,g_1)$ and $\neg R(g_0',g_1)$ and let $f(x_0,\dots,x_n):= R(x_0^0,x_1^0)$. Clearly, the tuples $(g_0,g_1,\dots,g_n)$ and $(g'_0,g_1,\dots,g_n)$ have the same order type but $f(a_{g_0},\dots,a_{g_n})\neq f(a_{g'_0},\dots,a_{g_n}) $.
    \end{proof}

    \begin{claim}
        For $n'>1$, the sequence $(a_g)_{g\in G_2}:=(\overbrace{g,\dots,g}^{n'},n_0)_{g\in G_2}$ is $G_2$-indiscernible but it is not order-indiscernible with respect to the family of functions $\Psi^{n'+1}_{\mathcal{F}_{X/E}}$.
    \end{claim}

    \begin{proof}[Proof of the second claim]
        We show it for $n'=2$. 
        Let $f(y,z,x):=R(y,z)$. Then the formula $$\psi_f(y_0z_0x_0,y_1z_1x_1,y_2z_2y_2):=\min_{\sigma\in Sym(2)} f(y_{\sigma(0)},z_{\sigma(1)},x_{\sigma(2)})$$ belongs to $\Psi^{3}_{\mathcal{F}_{X/E}}$. However, taking  $g_0'<g_0<g_1<g_2$ such that \begin{itemize}
            \item $R(g_0,g_1)$, $R(g_0,g_2)$, $R(g_1,g_2)$
            \item $R(g'_0,g_2)$ and  $\neg R(g'_0,g_1)$
        \end{itemize}
        we have that the tuples $(g_0,g_1,g_2)$ and $(g_0',g_1,g_2)$ are order-isomorphic and $$\psi_f(g_0g_0n_0,g_1g_1n_0,g_2g_2n_0)\neq \psi_f(g'_0g'_0n_0,g_1g_1n_0,g_2g_2n_0).$$
    \end{proof}
\end{example}
\printbibliography
\end{document}